\documentclass[10pt]{article}
\usepackage{latexsym,amsfonts,amsthm,amsmath,epsfig,amssymb,cases}
\usepackage{mathrsfs}
\usepackage{enumerate}
\usepackage{graphicx}

\newtheorem{thm}{Theorem}[section]
\newtheorem{cor}{Corollary}[section]
\newtheorem{lemma}{Lemma}[section]

\newtheorem{assumption}{Assumption}[section]
\newtheorem{remark}{Remark}[section]

\newtheorem{defn}{Definition}[section]
\newtheorem{example}{Example}[section]
\newcounter{nextauthor}
\def\mathrm{\mbox}

\numberwithin{remark}{section}

\begin{document}
%\begin{CJK}{GBK}{kai}
\title{{\Large \bf  Strong solutions for jump-type stochastic differential equations with non-Lipschitz coefficients}\thanks{This work was supported by  the National Natural Science Foundation of China (11471230, 11671282).}}
\author{Zhun Gou, Ming-hui Wang and Nan-jing Huang\footnote{Corresponding author.  E-mail addresses: nanjinghuang@hotmail.com; njhuang@scu.edu.cn}\\
{\small\it Department of Mathematics, Sichuan University, Chengdu,	Sichuan 610064, P.R. China}}
\date{}
\maketitle
\begin{center}
\begin{minipage}{5.5in}
\noindent{\bf Abstract.} In this paper, the existence and pathwise uniqueness of strong solutions for jump-type stochastic differential equations are investigated under non-Lipschitz conditions.  A sufficient condition is obtained for ensuring the non-confluent property of strong solutions of jump-type stochastic differential equations. Moreover, some examples are given to illustrate our results.
\end{minipage}
\end{center}

\section{Introduction}
\paragraph{}
Jump-type stochastic differential equations (JSDEs), as natural extensions of stochastic differential equations (SDEs),  have been widely applied to many fields of science and engineering such as  physics, astronomy, finance, ecology, biology and so on. As for the applications in physics, Chudley and Elliott \cite{Chudley1961Neutron} applied JSDEs to describe atomic diffusion typically consists of jumps between vacant lattice sites. Bergquist et al.  \cite{ber} illustrated the quantum jumps in a single atom by JSDEs. Gleyzes et al. \cite{Gle} employed JSDEs to analyze the observation that the microscopic quantum system exhibits at random times sudden jumps between its states. Pellegrini \cite{pell} proved the existence and uniqueness of a solution for the jump-type stochastic Schr$\ddot{o}$dinger equations. As for the applications in finance, Shreve \cite{shreve} and Tankov \cite{tankov2003financial} have enumerated many financial models which can be described by JSDEs. Thus, it would be necessary to study some properties of solutions to JSDEs. In this paper, we mainly investigate some qualitative properties of solutions to JSDEs under non-Lipschitz conditions.

The linear growth condition guarantees that the solutions for JSDEs has no finite explosion time with probability one. However, the linear growth condition may not be satisfied in some practical situations. For instance, in the mathematical ecological models of \cite{khasminskii2001long, mao2002environmental}, the coefficients do not satisfy the linear growth condition while non-explosion is still guaranteed. Some non-explosive results for general SDEs without jumps under the linear growth condition can be found in \cite{Fang2003A, fang2003stochastic, Lan2014New}. Thus, one natural question is: can we relax the linear growth condition for JSDEs?  The first task of this paper is to provide a new sufficient super linear growth condition for ensuring the non-explosion of strong solutions for JSDEs.

In general, the usual method for studying the pathwise uniqueness of strong solutions for SDEs with Lipschitz conditions is to employ Gronwall's inequality to demonstrate that the distance $\mathbb{E}^{\frac{1}{2}}[|\widetilde{X}(t)-X(t)|^2]$ between two solutions $\widetilde{X}(t)$ and $X(t)$ vanishes \cite{ikeda2014stochastic}. Unfortunately, as pointed out by Fang and Zhang \cite{fang2005study}, the usual method employed in the previous literature is not applicable without the usual Lipschitz condition. In 1971, Yamada and Watanabe \cite{Yamada1971On} showed that the Lipschitz condition can be relaxed to the H\"{o}lder condition for one-dimensional SDEs. Recently, the pathwise uniqueness of strong solutions for SDEs with non-Lipschitz conditions has been studied by many authors (see, for example \cite{Lan2014New, shao2012harnack}). However, to the best of our knowledge, there are only a few papers dealing with the pathwise uniqueness of strong solutions for JSDEs with non-Lipschitz conditions (see \cite{FU2010306, Zenghu2009Strong}). The second task of this paper is to give a new non-lipschitz condition to guarantee the pathwise uniqueness of strong solutions to JSDEs.

On the other hand, the closely related non-confluent property (also known as the non-contact property) of strong solutions for SDEs with the Lipschitz condition has been studied by several authors (see, for example, \cite{emery1981non, uppman1982flot} and the references therein). Moreover, some sufficient conditions are derived for ensuring the non-confluent property of strong solutions for SDEs without jumps with non-lipschitz coefficients in \cite{fang2005study, Lan2014New}. However, the non-confluent property of strong solutions for SDEs with jumps had not been studied until a sufficient condition was established by Xi and Zhu \cite{Xi2017Jump}. The third task of this paper is to give a new sufficient condition for ensuring the non-confluent property of strong solutions for SDEs with jumps.

The rest of this paper is structured as follows. Section 2 presents some necessary preliminaries including assumptions and lemmas. In section 3, we obtain main results concerned with the non-explosion and pathwise uniqueness of strong solutions for JSDEs with super linear growth and non-Lipschitz conditions. Before concluding this paper, the non-confluent property of strong solutions for JSDEs is investigated in Section 4.

\section{Preliminaries}
\paragraph{}
Let $\{p_1(t)\}$ and $\{p_2(t)\}$ be two $\mathcal{F}_t$-Poisson point processes on $U_1$ and $U_2$ with characteristic measures $\nu_1(du)$ and $\nu_2(du)$, respectively, such that $\{B_t\}, \{p_1(t)\}, \{p_2(t)\}$ are independent of each other. Let $N_1(ds,du)$ and $N_2(ds,du)$ be Poisson random measures associated with $\{p_1(t)\}$ and $\{p_2(t)\}$, respectively. Moreover, suppose that $b:\mathbb{R}\rightarrow\mathbb{R}$ and $\sigma:\mathbb{R}\rightarrow\mathbb{R}$ are two continuous functions, $c_1:\mathbb{R}{\times}U_1\rightarrow\mathbb{R}$ and $c_2:\mathbb{R}{\times}U_2\rightarrow\mathbb{R}$ are two Borel functions.

In this paper, we consider the following JSDE:
\begin{equation}
\begin{split}
X(t)=X_0&\mbox{}+\int_{0}^{t}\sigma(X(s))dB_s+\int_{0}^{t}\int_{U_1}c_1(X(s-),u)\widetilde{N_1}(ds,du)\\
&\mbox{}+\int_{0}^{t}b(X(s))ds+\int_{0}^{t}\int_{U_2}c_2(X(s-),u){N_2}(ds,du)\label{eq-JSDE}
\end{split}
\end{equation}
with $\mathbb{E}[|X_0|^2]<\infty$, where
$$\widetilde{N_1}(dt,du)=N_1(dt,du)-\nu(dz)dt
$$
is the compensated Poisson random measure of $N_1(ds,dt)$. Since $X(s)\neq X(s-)$ for at most countably many $s>0$, we know that $\sigma(X(s-))$ and $b(X(s-))$ can be replaced by $\sigma(X(s))$ and $b(X(s))$, respectively.
\begin{defn}
A process ${X(t)}$ is said to be a strong solution of (\ref{eq-JSDE}) if it is $\mathscr{F}_t$-adapted almost surely for every $t\geq 0$, where $\mathscr{F}_t=\sigma(\{B_t\}, \{p_1(t)\}, \{p_2(t)\})$ is the augmented natural filtration generated by $\{B_t\}$, $\{p_1(t)\}$ and $\{p_2(t)\}$.
\end{defn}

\begin{lemma}(\cite{FU2010306})\label{dengjiajie}
Let $U_3$ be a subset of $U_2$ satisfying $\nu_2(U_2\backslash U_3)<\infty$ and consider the following JSDE:
\begin{eqnarray}
X(t)&=&X_0+\int_{0}^{t}\sigma(X(s))dB_s+\int_{0}^{t}\int_{U_1}c_1(X(s-),u)\widetilde{N_1}(ds,du)\nonumber\\
&&\mbox{}+\int_{0}^{t}b(X(s))ds+\int_{0}^{t}\int_{U_3}c_2(X(s-),u){N_2}(ds,du).\label{eq-SDE2}
\end{eqnarray}
Then (\ref{eq-JSDE}) has a strong solution if (\ref{eq-SDE2}) has a strong solution. Moreover, the pathwise uniqueness of strong solutions holds for (\ref{eq-JSDE}) if it holds for (\ref{eq-SDE2}).
\end{lemma}

In this paper,  we need the following assumptions.

\begin{assumption}\label{ass-c}
Suppose that there exists a non-decreasing, continuous and concave function
$\rho:[0,\infty)\rightarrow[0,\infty)$ such that $\rho(x)>0$ for $x>0$
satisfing
\begin{equation}
\int_{0+}\frac{ds}{\rho(s)}=\infty.\label{con:ass}
\end{equation}
\end{assumption}

Clearly, the following functions satisfy (\ref{con:ass}):
\begin{align*}
&\rho(x)=x(x>0);                            &\rho(x)&=-x\ln x(0<x\leq \frac{1}{e});\\
&\rho(x)=x\ln(-\ln x)(0<x\leq \frac{1}{e}); &\rho(x)&=1-x^x(0<x\leq \frac{1}{e}).
\end{align*}

\begin{assumption}\label{ass-b}
Assume that there exists a non-decreasing and continuously differentiable function $\Upsilon:[0,\infty)\rightarrow [1,\infty)$ satisfying
\begin{enumerate}[($\romannumeral1$)]
\item $\lim \limits_{x\rightarrow \infty}\Upsilon(x)=+\infty$;
\item $\int_{0}^{\infty}\frac{ds}{s\Upsilon(s)+1}=+\infty$;
\item $2xb(x)+|\sigma(x)|^2+\int_{U_1}|c_1(x-,u)|^2\nu_1(du)+2\int_{U_3}|c_2(x-,u)|^2\nu_2(du)\leq \mu[x^2\Upsilon(x^2)+1]$
for all $x\in \mathbb{R}$, where $\mu\geq0$ is a fixed constant.
\end{enumerate}
\end{assumption}

Clearly, the following functions satisfy Assumption \ref{ass-b}:
\begin{align*}
&\Upsilon(x)=1(x\geq0); &\Upsilon(x)&=\ln x(x\geq e);
&\Upsilon(x)=\ln x\ln(\ln x)(x\geq e^2).
\end{align*}

\begin{assumption}\label{ass-a}
Suppose that there exists a constant $\delta_0>0$ such that, for any $x,y\in\mathbb{R}$ with $0<|x-y|\leq{\delta_0}$,
\begin{enumerate}[($\romannumeral1$)]
\item $\max\left\{(x-y)(b(x)-b(y)),|\sigma(x)-\sigma(y)|^2\right\}{\leq}{|x-y|^{2-\alpha}}\rho(|x-y|^\alpha)$;
\item $\int_{U_1}\max\{|c_1(x,u)-c_1(y,u)|^\alpha ,|x-y|^{\alpha-1}\cdot |c_1(x,u)-c_1(y,u)|\}\nu_1(du)\leq \rho(|x-y|^\alpha)$;
\item $\int_{U_3}\max\{|c_2(x,u)-c_2(y,u)|^\alpha ,|x-y|^{\alpha-1}\cdot |c_2(x,u)-c_2(y,u)|\}\nu_2(du)\leq \rho(|x-y|^\alpha)$.
\end{enumerate}
Here $0\leq \alpha< +\infty$ is a fixed constant and $\rho$ is defined in Assumption \ref{ass-c}.
\end{assumption}

\begin{remark}
In particular, Assumption \ref{ass-a}  reduces to the Lipschitz case when $\alpha=0$. Thus it is sufficient to consider the situation for $0< \alpha< +\infty$.
\end{remark}

\begin{assumption}\label{as-cor}
Assume that there exist two non-decreasing, continuous functions
$\rho_1, \rho_2:[0,\infty)\rightarrow[0,\infty)$
satisfing:
\[\int_{0+}\frac{ds}{\rho_i(s)}=\infty,\;i=1,2,\]
where $\rho_1$ is concave. In addition, suppose that there exists a constant $\delta_0>0$ such that,  for any $x, y\in\mathbb{R}$ with $0<|x-y|\leq{\delta_0}$,
\begin{enumerate}[($\romannumeral1$)]
\item $(x-y)(b(x)-b(y))+\int_{U_3}|c_2(x,u)-c_2(y,u)|\nu_2(du)\leq {|x-y|}\rho_1(|x-y|)$,
\item $|\sigma(x)-\sigma(y)|^2+\int_{U_1}|c_1(x,u)-c_1(y,u)|^2\nu_1(du)\leq \rho_2(|x-y|)$.

\end{enumerate}
Here $c_1(x,u)$ is non-decreasing for each fixed $u$.
\end{assumption}

\begin{assumption}\label{ass-confluence}
Suppose that
\begin{align}
\nu_i(u\in U: & \; \mbox{there exist}  \; x,y\in\mathbb{R}\;with \;x\neq y \; \mbox{such that} \nonumber \\
& \mbox{} \; |x-y+c_i(x,u)-c_i(y,u)|\leq \delta |x-y|)=0, \; i=1,2, \label{condition nu}
\end{align}
where $\delta>0$ is a fixed constant. In addition, assume that, for any $x,y\in\mathbb{R}$,
\begin{enumerate}[($\romannumeral1$)]
\item $(x-y)(b(x)-b(y)){\leq}{|x-y|^{2+\alpha}}\rho(|x-y|^{-\alpha})$;
\item $|\sigma(x)-\sigma(y)|^2{\leq}{|x-y|^{2+\alpha}}\rho(|x-y|^{-\alpha})$;
\item $\int_{U_i}|c_i(x,u)-c_i(y,u)|\nu_i(du)\leq |x-y|^{1+\alpha}\rho(|x-y|^{-\alpha})$, $i=1,2$.
\end{enumerate}
Here $0\leq \alpha< +\infty$ is a fixed constant and $\rho$ is defined in Assumption \ref{ass-c}.
\end{assumption}
\begin{remark}
If $\alpha=0$, then Assumption \ref{ass-confluence} reduces to the assumption in Corollary 3.3 in \cite{Xi2017Jump}.
\end{remark}

In order to obtain our main results, we also need the following lemmas.

\begin{lemma} (\cite{?2007Applied})\label{l2.1}
Suppose that $X(t)\in\mathbb{R}$ is an It\^{o}-L\'{e}vy process of the following form:
$$
dX(t)=b(t,\omega)dt+\sigma(t,\omega)dB_t+\int_{\mathbb{R}}\gamma(t,u,\omega)\overline{N}(dt,du),
$$
where
\begin{equation}
\overline{N}(x)=
\begin{cases}
N(dt,du)-\nu(du)dt, &\mbox{if $|u|<R$};\\
N(dt,du), &\mbox{if $|u|\geq R$}
\end{cases}
\end{equation}
for some $R\in [0,+\infty)$. Let $f\in C^2(\mathbb{R}^2)$ and define $Y(t)=f(t,X(t))$. Then $Y(t)$ is again an It$\widehat{o}$-L$\acute{e}$vy process and
\begin{align*}
dY(t)&=\frac{\partial f}{\partial t}(t,X(t))dt+\frac{\partial f}{\partial x}(t,X(t))\left[b(t,\omega)dt+\sigma(t,\omega)dB_t\right]+\frac{1}{2}\sigma^2(t,X(t))\frac{\partial^2 f}{\partial x^2}(t,X(t-))\\
&\quad \mbox{}+\int_{|z|<\mathbb{R}}\bigg\{f(t,X(t-)+\gamma(t,u))-f(t,X(t-))-\frac{\partial f}{\partial x}(t,X(t-))\gamma(t,u)\bigg\}\nu(du)dt\\
&\quad \mbox{}+\int_{\mathbb{R}}\bigg\{f(t,X(t-)+\gamma(t,u))-f(t,X(t-))\bigg\}\overline{N}(dt,du).
\end{align*}
\end{lemma}

In the sequel, for any $f\in C^n(\mathbb{R})$, we will replace $\frac{\partial ^n}{\partial^n x}f(x)$ by $D^{(n)}f(x)$ for convenience.

\begin{lemma}\label{th-bihari}
Let $u(t)$ and $g(t)$ be non-negative continuous functions, and $f(t)$ a non-negative continuously differentiable and non-decreasing function for all $t\ge 0$.   Furthermore, suppose that $\rho: [0, +\infty)\to [0,+\infty)$ is a non-negative and non-decreasing continuous function with
$$
\rho(t)=0 \Longleftrightarrow t=0 \; \mbox{and} \; \int_{0+}\frac{ds}{\rho(s)}=\infty.
$$
Then the inequality
$$
u(t){\leq} f(t)+\int_{0}^{t}g(s)\rho(u(s))ds
$$
implies the inequality
$$
u(t){\leq} \Omega^{-1}\left[\Omega(f(t))+\int_{0}^{t}g(s)ds\right],
$$
where
$$
\Omega(t)=\int_{0}^{t}\frac{ds}{\rho(s)}, \quad \forall t>0.
$$
Moreover, if $f(t)=0$ and $|g(t)|<+\infty$, then $u(t)=0$.
\end{lemma}
\begin{proof}
Let
$$
v(t)=f(t)+\int_{0}^{t}g(s)\rho(u(s))ds
$$
and
$$
\phi(t)=\Omega(v(t))-\Omega(f(t))-\int_{0}^{t}g(s)ds.
$$
 Then $\max\{{u(t),f(t)}\}\leq v(t)$. By direct computations, we have
\begin{align*}
\phi'(t)&\mbox{}=\frac{f'(t)+g(t)\rho(u(t))}{\rho(v(t))}-\frac{f'(t)}{\rho(f(t))}-g(t)\\
&\mbox{}=\frac{\rho(f(t))-\rho(v(t))}{\rho(f(t))\cdot\rho(v(t))}f'(t)+g(t)\frac{\rho(u(t))-\rho(v(t))}{\rho(v(t))}\\
&\mbox{} \leq 0.
\end{align*}
This shows that $\phi(t)$ is non-increasing and
$$
\phi(t)=\Omega\left(f(t)+\int_{0}^{t}g(s)\rho(u(s))ds\right)-\Omega(f(t))-\int_{0}^{t}g(s)ds{\leq}\phi(0)=0.
$$
Moreover, since $\Omega(t)$ is increasing, one has
$$
\Omega(u(t)){\leq}\Omega\left(f(t)+\int_{0}^{t}g(s)\rho(u(s))ds\right){\leq}\Omega(f(t))+\int_{0}^{t}g(s)ds
$$
and so
$$
u(t){\leq} \Omega^{-1}\left(\Omega(f(t))+\int_{0}^{t}g(s)ds\right).
$$

On the other hand, define $\Psi(t)=\int_{1}^{t}\frac{ds}{\rho(s)}$. Then $\Psi(t)$ is increasing and satisfies $\Psi(0)=-\infty$ since $\int_{0+}\frac{ds}{\rho(s)}=\infty$. Letting $f(t)=0$, it follows that
$$
\Psi(u(t))\leq \Psi(v(t))=G(0)+\int_{0}^{t}\Psi'(v(s))v'(s)ds=\Psi(0)+\int_{0}^{t}\frac{\Psi(u(s))}{\Psi(v(s))}g(s)ds\leq \Psi(0)+\int_{0}^{t}g(s)ds.
$$
Since $\Psi(0)=-\infty$ and $\int_{0}^{t}g(s)ds<+\infty$, we have $\Psi(u(t))=-\infty$ and $u(t)=0$  consequently.
\end{proof}

\begin{remark}
If $f(t)=k>0$ is a constant function, then Lemma \ref{th-bihari} reduces to the corresponding result in \cite{bihari1956generalization}.
\end{remark}

\section{Non-explosion and Pathwise Uniqueness}
\paragraph{}

\begin{thm}\label{thm-explode}
Under Assumption \ref{ass-b}, the solutions for JSDE (\ref{eq-JSDE}) do not explode in finite time.
\end{thm}
\begin{proof}
Define
$$
\phi(x)=\exp\left\{\int_{0}^{x}\frac{ds}{s\Upsilon(s)+1}\right\}, \quad x\geq 0.
$$
By simple computations, we have
$$
\phi'(x)=\phi(x)\frac{1}{x\Upsilon(x)+1}\geq0, \quad \phi''(x)=\phi(x)\frac{1-\Upsilon(x)-x\Upsilon'(x)}{(x\Upsilon(x)+1)^2}\leq0.
$$
Clearly, $\phi(x)$ is a concave function with $\phi(x)\rightarrow\infty$ as $x\rightarrow\infty$. Moreover,
$$
D^{(1)}\phi(x^2)=\phi'(x^2)\cdot 2x, \quad D^{(2)}\phi(x^2)=2\phi'(x^2)+\phi''(x^2)\cdot 4x^2.
$$
Since $\phi''(x)\leq0$, we know that
$$\phi(y)\leq \phi(x)+(y-x)\phi'(x), \quad \forall x,y\in[0,\infty).$$
It follows that
\begin{align*}
&\quad \mbox{} \phi\left((X(s-)+c_1\left(X(s-),u\right))^2\right)-\phi(X^2(s-))-D^{(1)}\phi\left((X^2(s-))\cdot c_1(X(s-),u)\right)\\
&\mbox{}\leq \phi'(X^2(s-))\left[2X(s-)c_1(X(s-))+c_1^2\left(X(s-),u\right)\right]-\phi'(X^2(s-)).2X(s-)c_1\left(X(s-),u\right)\\
&\mbox{}=\phi'(X^2(s-))c_1^2(X(s-),u).
\end{align*}
Let
$$\tau_R:=\inf\left\{t\geq0:\max\left\{{|\widetilde{X}(t)|}\ , {|X(t)|}\right\}\geq{R}\right\}.$$  Applying Lemma \ref{l2.1}, one has
\begin{align*}
\mathbb{E}\left[\phi(X^2(t\wedge{\tau_R}))\right]&\mbox{}=\mathbb{E}[\phi(|X_0|^2)]+\mathbb{E}\left[\int_{0}^{t\wedge{\tau_R}}D^{(1)}\phi(X^2(s))b(X(s))+\frac{1}{2}\sigma^2(X(s)).D^{(2)}\phi(X^2(s))ds\right]\\
&\quad \mbox{}+\mathbb{E}\bigg[\int_{0}^{t\wedge{\tau_R}}\int_{U_1}\phi(|X(s-)+c_1(X(s-),u)|^2)-\phi(X^2(s-))\\
&\quad  \mbox{} -D^{(1)}\phi(|X^2(s-)|)\cdot c_1(X(s-),u)\nu_1(du)ds\bigg]\\
&\quad \mbox{} +\mathbb{E}\left[\int_{0}^{t\wedge{\tau_R}}\int_{U_3}\phi\left((X(s-)+c_2(X(s-),u))^2\right)-\phi((X^2(s-))\nu_2(du)ds\right]\\
&\mbox{}\leq \mathbb{E}[\phi(|X_0)|^2]+\mathbb{E}\left[\int_{0}^{t\wedge{\tau_R}}\phi'(X^2(s))\cdot \big[2X(s)b(X(s))+\sigma^2(X(s))\big]\right.\\
&\quad \left. \mbox{} +2(X^2(s))\phi''(X^2(s))\sigma^2(X(s))ds\right.\\
&\quad \left. \mbox{} +\int_{U_1}\phi'(X^2(s-))\cdot|c_1(X(s-),u)|^2\nu_1(du)ds\right]\\
&\quad \mbox{}+\mathbb{E}\left[\int_{0}^{t\wedge{\tau_R}}\int_{U_3}\phi'(X^2(s-))\cdot(2|c_2(X(s-),u)|^2+X^2(s-))\nu_2(du)ds\right].
\end{align*}
Since $\phi''(x)\leq0$ and $\int_{U_3}1\nu_2(du)\leq \int_{U_2}1\nu_2(du)\leq M$, by Assumption \ref{ass-b}, we have
\begin{align*}
\mathbb{E}\left[\phi(X^2(t\wedge{\tau_R}))\right]\leq &\mbox{} \phi\left(\mathbb{E}[|X_0|^2]\right)+\mathbb{E}\left[\int_{0}^{t\wedge{\tau_R}}\phi'(X^2(s))\cdot (M+1)\mu[X^2(s)\Upsilon(X^2(s))+1]ds\right]\\
=&\mbox{} \phi\left(\mathbb{E}[|X_0|^2]\right)+(M+1)\mu \int_{0}^{t}\mathbb{E}[\phi(X^2(s\wedge{\tau_R}))]ds.
\end{align*}
Thus, it follows from Gronwall's inequality that
$$
\mathbb{E}\left[\phi(X^2(t\wedge{\tau_R}))\right]\leq \phi\left(\mathbb{E}[|X_0|^2]\right)\cdot e^{\mu(M+1)t}.
$$
Letting $|X(t\wedge{\tau_R})|\rightarrow \infty$, we have $t\rightarrow\infty$ as $\mathbb{E}[|X_0|^2]<\infty$. Therefore,  the solution has no finite explosion time.
\end{proof}

\begin{thm}\label{th-unique}
Under Assumptions \ref{ass-b} and \ref{ass-a}, the pathwise uniqueness of strong solutions for JSDE (\ref{eq-SDE2}) holds.
\end{thm}
\begin{proof}
By the assumptions imposed on $\rho$, we can find a strictly decreasing sequence $\{a_n\}\subset(0,1]$ such that
\begin{enumerate}[($\romannumeral1$)]
\item $a_0=1$;
\item $\lim\limits_{n\rightarrow\infty} a_n=0$;
\item $\int_{a_n}^{a_{n-1}}\frac{1}{\rho{(r)}}dr=n$ for every $n\geq1$.
\end{enumerate}
Clearly, for each $n\geq1$, there exists a continuous function $\rho_n$ on $\mathbb{R}$ such that
\begin{enumerate}[($\romannumeral1$)]
\item $\rho_n(r)$ has a supported set $(a_n,a_{n-1})$;
\item $0\leq{\rho_n{(r)}}\leq \frac{2}{n\rho(r)}$ for every $r>0$;
\item $\int_{a_n}^{a_{n-1}}\rho_n(r)dr=1$.
\end{enumerate}
Now we consider the following sequence of functions:
$$
\psi_n(r)=\int_{0}^{|r|}\int_{0}^{v}\rho_n(u)dudv,\;r\in\mathbb{R},\;n\geq1.
$$
Clearly, $\psi_n$ is even and twice continuously differentiable (except at $r=0$) with the following properties:
\begin{enumerate}[($\romannumeral1$)]
 \item $|\psi'_n(r)|\leq1, \quad r\neq0$;
 \item $\lim\limits_{n\rightarrow\infty}\psi_n(r)=|r|, \quad r\neq0$;
 \item $\psi''_n(r)\leq{\frac{2}{n\rho(r)}I_{(a_n,a_{n-1})}(r)}, \quad r\neq0$.
 \end{enumerate}
Furthermore, for each $r>0$, the sequence $\{\psi_n(r)\}_{n\geq1}$ is non-decreasing. Note that for each $n\in\mathbb{N}$, $\psi_n$, $\psi'_n$ and $\psi''_n$ all vanish on the interval $(-a_n,a_n)$. By direct computations, we have, for $0\neq{x}\in\mathbb{R}$,
$$
D\psi_n(|x|^\alpha)=\frac{d}{dx}\psi_n(|x|^\alpha)=\psi'_n(|x|^\alpha)\cdot \alpha x\cdot |x|^{\alpha-2}
$$
and
\begin{equation*}
D^2\psi_n(|x|^\alpha)=\psi''_n(|x|^\alpha)\cdot \alpha^2 |x|^{2\alpha-2}+\psi'_n(|x|^\alpha)\cdot \alpha(\alpha-1) |x|^{\alpha-2} \label{D^2}.
\end{equation*}

Next we suppose that $X$ and $\widetilde{X}$ are two solutions for (\ref{eq-SDE2}) of the following forms:
\begin{equation*}
\begin{split}
X(t)=X_0&+\int_{0}^{t}\sigma(X(s))dB_s+\int_{0}^{t}\int_{U_1}c_1(X(s-),u)\widetilde{N_1}(ds,du)\\
&+\int_{0}^{t}b(X(s))ds+\int_{0}^{t}\int_{U_3}c_2(X(s-),u){N_2}(ds,du)\label{solution1}
\end{split}
\end{equation*}
and
\begin{equation*}
\begin{split}
\widetilde{X}(t)=X_0&+\int_{0}^{t}\sigma(\widetilde{X}(s))dB_s+\int_{0}^{t}\int_{U_1}c_1(\widetilde{X}(s-),u)\widetilde{N_1}(ds,du)\\
&+\int_{0}^{t}b(\widetilde{X}(s))ds+\int_{0}^{t}\int_{U_3}c_2(\widetilde{X}(s-),u){N_2}(ds,du)\label{solution2}
\end{split}
\end{equation*}
for all $t\geq0$, where $x,\widetilde{x}\in\mathbb{R}$.

Denote $\Delta_t:=\widetilde{X}(t)-X(t)$ for all $t\geq0$ and define
$$
S_{\delta_0}=\inf\left\{t\geq0:|\Delta_t|\geq\delta_0\right\}=\inf\left\{t\geq0:|\widetilde{X}(t)-X(t)|\geq\delta_0\right\}.
$$
For $R>0$, let
$$
\tau_R:=\inf\left\{t\geq0:\max\left\{{|\widetilde{X}(t)|}\ ,{|X(t)|}\right\}\geq{R}\right\}.
$$
Then, by Theorem \ref{thm-explode}, we have $\tau_R\rightarrow\infty$ a.s. as $R\rightarrow\infty$. Denote $t'=t\wedge{\tau_R}\wedge{S_{\delta_0}}$ and
$$
\Delta_{c_i}=c_i(\widetilde{X}(s-),u)-c_i(X(s-),u), \quad i=1,2.
$$
Applying Lemma \ref{l2.1}, we have
\begin{align*}
\mathbb{E}\left[\psi_n(|\Delta_{t'}|^\alpha)\right]&=\mathbb{E}\left[\int_{0}^{t'}I_{\{\Delta_s\neq0\}}\bigg\{D\psi_n(|\Delta_s|^\alpha)(b(\widetilde{X}(s))-b(X(s)))\right.\\
&\quad \left. \mbox{} +\frac{1}{2}D^2\psi_n(|\Delta_s|^\alpha)|\sigma(\widetilde{X}(s))-\sigma(X(s))|^2ds\bigg\}\right]\\
&\quad \mbox{}+\mathbb{E}\left[\int_{0}^{t'}\int_{U_1}\{\psi_n(|\Delta_s+
\Delta_{c_1}|^\alpha)-\psi_n(|\Delta_s|^\alpha)-I_{\{\Delta_s\neq0\}}D\psi_n(|\Delta_s|^\alpha)\cdot\Delta_{c_1}\}\nu_1(du)ds\right.\\
&\quad \left. \mbox{} +\int_{0}^{t'}\int_{U_3}\{\psi_n(|\Delta_s+\Delta_{c_2}|^\alpha)-\psi_n(|\Delta_s|^\alpha)\}\nu_2(du)ds\right]\\
&=J_1+J_2.
\end{align*}
Since
$$
|\psi'_n(r)|\leq{1}, \quad \psi''_n(r)\leq{\frac{2}{n\rho(r)}I_{(a_n,a_{n-1})}(r)},
$$
it follows from Assumption \ref{ass-a} that
\begin{align*}
J_1&\leq \mathbb{E}\Bigg[\int_{0}^{t'}I_{\{\Delta_s\neq0\}}\bigg\{|\psi'_n(|\Delta_s|^\alpha)|\cdot|\alpha|\cdot|\Delta_s|^{\alpha-2}(\widetilde{X}(s)-X(s))
(b(\widetilde{X}(s))-b(X(s)))\\
&\quad \mbox{}+\frac{1}{2}\left\{|\psi''_n(|\Delta_s|^\alpha)|\cdot \alpha^2|\Delta_s|^{2\alpha-2}+|\alpha(\alpha-1)|\cdot\psi'_n(|\Delta_s|^\alpha)\cdot |\Delta_s|^{\alpha-2}\right\}|\sigma(\widetilde{X}(s))-\sigma(X(s))|^2ds\bigg\}\Bigg]\\
&\mbox{} \leq \mathbb{E}\Bigg[\int_{0}^{t'}I_{\{\Delta_s\neq0\}}\bigg[{|\alpha|\cdot|\Delta_s|^{\alpha-2}{|\Delta_s|^{2-\alpha}}\rho(|\Delta_s|^\alpha)}\\
&\quad \mbox{}+\frac{1}{2}\left\{\frac{2}{n\rho(|\Delta_s|^\alpha)}I_{(a_n,a_{n-1})}(|\Delta_s|^\alpha)\alpha^2|\Delta_s|^{2\alpha-2}+
|\alpha(\alpha-1)|\cdot |\Delta_s|^{\alpha-2}\right\}{|\Delta_s|^{2-\alpha}}{\rho(|\Delta_s|^\alpha)}\bigg]ds\Bigg]\\
&\mbox{}\leq \mathbb{E}\left[\int_{0}^{t'}(\frac{1}{2}|\alpha(\alpha-1)|+|\alpha|)\cdot{\rho(|\Delta_s|^\alpha)}
\mbox{}+\frac{\alpha^2|\Delta_s|^\alpha}{n}I_{(a_n,a_{n-1})}(|\Delta_s|^\alpha)ds\right]\\
&\mbox{} \leq \left(\frac{1}{2}|\alpha(\alpha-1)|+|\alpha|\right)\mathbb{E}\bigg[\int_{0}^{t'}{\rho(|\Delta_s|^\alpha)}ds\bigg]
+\frac{\alpha^2a_{n-1}^{\alpha}}{n}t'.
\end{align*}
Regarding $J_2$, by Lagrange's mean value theorem and the fact that
$|\psi'_n(r)|\leq{1}$, we have the following cases:

Case I. For $0<\alpha\leq1$, since $(A+B)^\alpha{\leq}A^\alpha+B^\alpha$ for all $A,B\geq0$, we know that there exists some $\xi_1\in[|\Delta_s|^\alpha,(|\Delta_s|+|\Delta_{c_i}|)^\alpha]$ such that
\begin{align*}
\psi_n(|\Delta_s+\Delta_{c_i}|^\alpha)-\psi_n(|\Delta_s|^\alpha)&\leq \psi_n(|\Delta_s|^\alpha+|\Delta_{c_i}|^\alpha)-\psi_n(|\Delta_s|^\alpha)\\
&\leq|\psi'_n(\xi_1)|\cdot ||\Delta_s|^\alpha+|\Delta_{c_i}|^\alpha-|\Delta_s|^\alpha|\\
&\leq|\Delta_{c_i}|^\alpha.
\end{align*}

Case II. For $1<\alpha<+\infty$, since $(A+B)^{\alpha-1}\leq (2^{\alpha-2}+1)(A^{\alpha-1}+B^{\alpha-1})$ for all $A,B\geq0$, there exists some $\xi_2\in[\min\{\Delta_s,\Delta_s+\Delta_{c_i}\}, \max\{\Delta_s,\Delta_s+\Delta_{c_i}\}]$ such that
\begin{align*}
\psi_n(|\Delta_s+\Delta_{c_i}|^\alpha)-\psi_n(|\Delta_s|^\alpha)&\mbox{}\leq \alpha|\psi'_n(|\xi_2|^\alpha)|\cdot |\xi_2|^{\alpha-1}\cdot |\Delta_s+\Delta_{c_i}-\Delta_s|\\
&\mbox{}\leq \alpha(|\Delta_s|+|\Delta_{c_i}|)^{\alpha-1}|\Delta_{c_i}|\\
&\mbox{}\leq \alpha(2^{\alpha-2}+1)(|\Delta_s|^{\alpha-1}|\Delta_{c_i}|+|\Delta_{c_i}|^{\alpha}),
\end{align*}
where the second inequality follows from
$$
0\leq |\xi_2|\leq \max\{|\Delta_s|,|\Delta_s+\Delta_{c_i}|\}\leq |\Delta_s|+|\Delta_{c_i}|.
$$
Thus,
\begin{align*}
J_2
&\leq \mathbb{E}\left[\int_{0}^{t'}[\alpha(2^{\alpha-2}+1)+1]{\rho(|\Delta_s|^\alpha)}+\alpha{\rho(|\Delta_s|^\alpha)}+[\alpha(2^{\alpha-2}+1)+1]{\rho(|\Delta_s|^\alpha)}\right]ds\\
&\leq [\alpha(2^{\alpha}+3)+2]\cdot\mathbb{E}\left[\int_{0}^{t'}{\rho(|\Delta_s|^\alpha)}ds\right]
\end{align*}
and so
$$
\mathbb{E}\left[\psi_n(|\Delta_{t'}|^\alpha)\right]\leq p(\alpha)\mathbb{E}\left[\int_{0}^{t'}{\rho(|\Delta_s|^\alpha)}ds\right]+\frac{\alpha^2a_{n-1}^{\alpha}}{n}t,
$$
where
$$
p(\alpha)=\frac{1}{2}|\alpha(\alpha-1)|+|\alpha|+\alpha(2^{\alpha}+3)+2.
$$
Since $\lim \limits_{n\rightarrow\infty}\psi_n(r)=|r|$, letting $n\rightarrow\infty$ yields
\begin{align*}
\mathbb{E}\left[|\Delta_{t'}|^\alpha\right]&\leq p(\alpha)\mathbb{E}\bigg[\int_{0}^{t'}{\rho(|\Delta_s|^\alpha)}ds\bigg]\\
&\leq p(\alpha)\mathbb{E}\bigg[\int_{0}^{t\wedge{\tau_R}}{\rho\left(|\Delta_{s\wedge{S_{\delta_0}}}|^\alpha\right)}ds\bigg]\\
&\leq p(\alpha)\int_{0}^{t}{\rho\left(\mathbb{E}(|\Delta_{s\wedge{S_{\delta_0}}\wedge{\tau_R}}|^\alpha)\right)}ds,
\end{align*}
where the last inequality follows from Jensen's inequality.  It follows from Theorem \ref{thm-explode}, Fatou's lemma and the monotone convergence theorem that
$$
\mathbb{E}[|\Delta_{t\wedge{S_{\delta_0}}}|^\alpha]\leq \lim_{R\rightarrow\infty}\mathbb{E}[|\Delta_{t'}|^\alpha]
\leq p(\alpha)\int_{0}^{t}{\rho\left(\mathbb{E}(|\Delta_{s\wedge{S_{\delta_0}}}|^\alpha)\right)}ds.
$$
Applying Lemma \ref{th-bihari} yields that $\mathbb{E}[|\Delta_{t\wedge{S_{\delta_0}}}|^\alpha] \rightarrow0$ and so $\Delta_{t\wedge{S_{\delta_0}}}=0$ a.s..

On the set $\{S_{\delta_0}\leq t\}$, we have $|\Delta_{t'}|\geq \delta_0$. Observing that $0=\mathbb{E}[|\Delta_{t\wedge{S_{\delta_0}}}|^{\alpha}]\geq \delta_0^{\alpha}\mathbb{P}\{S_{\delta_0}\leq t\}$,
we have $\mathbb{P}\{S_{\delta_0}\leq t\}=0$ and hence $\Delta_t=0$  a.s., which is the desired result.
\end{proof}

\begin{remark}\label{tech}
We would like to point out that the proof method of Theorem \ref{th-unique} is similar to the one of Theorem 2.4 in \cite{Xi2017Jump}.
\end{remark}

\begin{thm}\label{existence}
Under Assumptions \ref{ass-b} and \ref{ass-a}, JSDE (\ref{eq-JSDE}) has a unique non-explosive strong solution.
\end{thm}

\begin{proof}
Similar to the proof of Theorem 2.2 in \cite{Zenghu2009Strong}, applying Theorems \ref{thm-explode} and \ref{th-unique}, we know that there exists a unique non-explosive strong solution for (\ref{eq-SDE2}). Thus, by Lemma \ref{dengjiajie}, there also exists a unique strong non-explosive solution for (\ref{eq-JSDE}).
\end{proof}

\begin{cor}\label{alpha=1 cor}
Under Assumptions \ref{ass-b} and \ref{as-cor}, JSDE (\ref{eq-JSDE}) has a unique non-explosive strong solution.
\end{cor}

\begin{proof}
Similar to the proof of Theorem \ref{th-unique}, replacing $\rho(r)$ by $\rho_2(r)$, we have
$$
D^2\psi_n(|\Delta_s|)|\sigma(\widetilde{X}(s))-\sigma(X(s))|^2\leq \frac{2}{n\rho_2(|\Delta_s|)}I_{(a_n,a_{n-1})}(|\Delta_s|)\cdot{\rho_2(|\Delta_s|)}=\frac{2}{n}I_{(a_n,a_{n-1})}.
$$
For convenience, we denote $\Delta_{c_1}=c_1(\widetilde{X}(s-),u)-c_1(X(s-),u)$. By Taylor's expansion, there exists some $\eta=\Delta_s+\theta \Delta_{c_1}$ with a constant $\theta\in(0,1)$ such that
\begin{eqnarray}\label{+1}
&& \int_{U_1}\left[\psi_n(|\Delta_s+\Delta_{c_1}|)-\psi_n(|\Delta_s|)-I_{\{\Delta_s\neq0\}}D\psi_n(|\Delta_s|)\cdot(\Delta_{c_1})\right]\nu_1(du)\nonumber\\
&=&\int_{U_1}\frac{1}{2}D^2\psi_n(|\eta|)\cdot |\Delta_{c_1}|^2\nu_1(du).
\end{eqnarray}
Since $c_1(x,u)$ is non-decreasing for fixed $u$, we have $\Delta_s\cdot\Delta_{c_1}\geq0$ and $|\eta|\geq|\Delta_s|$.  Thus, it follows from \eqref{+1} that
\begin{align*}
&\quad \mbox{}\int_{U_1}\left[\psi_n(|\Delta_s+\Delta_{c_1}|)-\psi_n(|\Delta_s|)-I_{\{\Delta_s\neq0\}}\cdot D\psi_n(|\Delta_s|)\cdot(\Delta_{c_1})\right]\nu_1(du)\\
&\mbox{}\leq \frac{1}{2}\cdot\frac{2}{n\rho_2(|\Delta_s|)}\int_{U_1}|\Delta_{c_1}|^2I_{(a_n,a_{n-1})}(|\eta|)\nu_1(du)\\
&\mbox{}\leq \frac{I_{(a_n,a_{n-1})}}{n}.
\end{align*}

The rest proof can be completed by the similar arguments to Theorems \ref{th-unique} and \ref{existence}, and so we omit it here.
\end{proof}

\begin{remark}
We would like to mention that Corollary \ref{alpha=1 cor} can be easily extended to the multi-dimensional case (see Theorem 3.3 of Fu and Li \cite{FU2010306}).
\end{remark}

\begin{example}
Consider the following SDE:
\begin{align}\label{exam.}
X(t)=&X_0-\int_{0}^{t}|X(s)|\ln|X(s)|ds+\int_{0}^{t}\sqrt{|X(s)|}dB_s+\int_{0}^{t}\int_{|u|\leq1}\sqrt{|X(s)|}\widetilde{N_1}(ds,du)\nonumber\\
&+\int_{0}^{t}\int_{|u|>1}\gamma|u| X(s){N_2}(ds,du).
\end{align}
Here $\gamma$ is a positive constant such that $\int_{|u|\leq1}|\gamma u|^2\nu(du)=1$.
It is easy to show that, for any $x>0$, the coefficient $b(x)=-x\cdot \ln x$ satisfies Assumptions \ref{ass-c} and \ref{ass-b}, and for any $x\geq0$, the coefficient $\sigma(x)=\sqrt x$ satisfies Assumption \ref{ass-b}. Thus, $b(x)$ and $\sigma(x)$ are both non-Lipschitzian due to
$$\lim \limits_{x\rightarrow 0^+}b'(x)=\lim \limits_{x\rightarrow 0^+}\sigma'(x)=+\infty.$$
Furthermore,
\begin{center}
$|b(x)-b(y)|=|\int_{x}^{y}1+\ln t\;dt|\leq \int_{0}^{|x-y|}|1+\ln t|\;dt=b(|x-y|)$
\end{center}
for all $0<x,y<\frac{1}{e}$, and
 \begin{center}
 $(\sigma(x)-\sigma(y))^2\leq |x-y|$
\end{center}
for all $x,y\geq0$. Thus,  the coefficients of (\ref{exam.}) satisfy Assumptions \ref{ass-b} and \ref{as-cor}. By Corollary \ref{alpha=1 cor}, we know that (\ref{exam.}) has a unique non-explosive strong solution.
\end{example}

\section{Non-confluent Property}
\paragraph{}
In this section, we present the non-confluent property of strong solutions to (\ref{eq-JSDE}).
\begin{defn}
Suppose that (\ref{eq-JSDE}) has a unique non-explosive strong solution $X(t)$ for any initial value $X_0$. Then $X(t)$ is said to have the non-confluent property if,  for any $x,y\in\mathbb{R}$ with $x\not=y$,
\[\mathbb{P}\{X^x(t)\neq X^y(t),\;for\;all\;t\geq0\}=1,\]
where $X^x$ and $X^y$ denote the strong solutions of (\ref{eq-JSDE}) with initial conditions $x$ and $y$, respectively.
\end{defn}

\begin{thm}\label{nonconfluence}
Suppose that Assumption \ref{ass-confluence} holds and (\ref{eq-JSDE}) has a unique non-explosive strong solution $X(t)$ for any initial value $X(0)=x \in \mathbb{R}$.
Then $X(t)$ has the non-confluent property.
\end{thm}

\begin{proof}
Consider the function $R(x)=|x|^{-\alpha}$. For any $x,y\in \mathbb{R}$ with $x\neq0$ and $|x+y|\geq \delta|x|$, where $\delta$ is a fixed constant. We now claim that
\begin{equation}\label{R(x)}
R(x+y)-R(x)-DR(x)\cdot y=\frac{|x|^{\alpha}-|x+y|^{\alpha}}{|x+y|^{\alpha}\cdot |x|^{\alpha}}-DR(x)\cdot y\leq K\frac{|x|^{\alpha}+|x|^{\alpha-1}|y|}{|x|^{2\alpha}},
\end{equation}
where $K>0$ is a constant.   Indeed, it is enough to see that
$$
-DR(x)\cdot y=\alpha x|x|^{-\alpha-2}y\leq \alpha\frac{|x|^{\alpha-1}|y|}{|x|^{2\alpha}}.
$$
For $|x|\leq |x+y|$,  (\ref{R(x)})  is automatically satisfied.  Thus,  it is sufficient to consider the case for $|x|\geq |x+y|$. For $0\leq \alpha\leq1$,
$$
|x|^{\alpha}-|x+y|^{\alpha}\leq |x|^{\alpha}.
$$
For $1< \alpha< +\infty$, there exists some $\xi\in (|x+y|,|x|)$ such that
$$|x|^{\alpha}-|x+y|^{\alpha}=\alpha|\xi|^{\alpha-1}(|x|-|x+y|)\leq \alpha|x|^{\alpha-1}|y|.$$
Since $|x+y|\geq \delta|x|$, (\ref{R(x)}) holds for $K=\frac{1}{\delta^\alpha}(1+2\alpha)$.

For any $x,\widetilde{x}\in \mathbb{R}$ with $x\not= \widetilde{x}$, let $X(t)$ and $\widetilde{X}(t)$ be two strong solutions for (\ref{eq-JSDE}) of the following forms:
\begin{equation*}
\begin{split}
X(t)=x&+\int_{0}^{t}\sigma(X(s))dB_s+\int_{0}^{t}\int_{U_1}c_1(X(s-),u)\widetilde{N_1}(ds,du)\\
&+\int_{0}^{t}b(X(s))ds+\int_{0}^{t}\int_{U_2}c_2(X(s-),u){N_2}(ds,du),
\end{split}
\end{equation*}
\begin{equation*}
\begin{split}
\widetilde{X}(t)=\widetilde{x}&+\int_{0}^{t}\sigma(\widetilde{X}(s))dB_s+\int_{0}^{t}\int_{U_1}c_1(\widetilde{X}(s-),u)\widetilde{N_1}(ds,du)\\
&+\int_{0}^{t}b(\widetilde{X}(s))ds+\int_{0}^{t}\int_{U_2}c_2(\widetilde{X}(s-),u){N_2}(ds,du).
\end{split}
\end{equation*}
Denote $\Delta_t=\widetilde{X}(t)-X(t)$. Then $|\Delta_0|=|x-\widetilde{x}|>0$.  For any $\frac{1}{|\Delta_0|}<n\in \mathbb{N}$ and $R>\max\{|\widetilde{x}|,|x|\}$, define
\begin{align*}
T_{\frac{1}{n}}:&=\inf\left\{t\geq0:|\Delta_t|\leq \frac{1}{n}\right\},\\
T_0:&=\inf\{t\geq0:|\Delta_t|=0\},\\
\tau_R:&=\inf\left\{t\geq0:\max\{{|\widetilde{X}(t)|}\ ,{|X(t)|}\}\geq{R}\right\}.
\end{align*}
Obviously,  $T_0=\lim \limits_{n\rightarrow\infty}T_{\frac{1}{n}}$ and $\lim \limits_{R\rightarrow\infty}\tau_R=\infty$ a.s..

Let $t'=t\wedge{\tau_R}\wedge{T_{\frac{1}{n}}}$ and
$$
\Delta_{c_i}=c_i(\widetilde{X}(s-),u)-c_i(X(s-),u), \quad i=1,2.
$$
It follows from (\ref{condition nu}), \eqref{R(x)} and Lemma \ref{l2.1} that
\begin{align*}
\mathbb{E}\left[|\Delta_{t'}|^{-\alpha}\right]&=|\Delta_0|^{-\alpha}+\mathbb{E}\Bigg[\int_{0}^{t'}-\alpha\Delta_{s}\cdot|\Delta_{s}|^{-\alpha-2}\cdot (b(\widetilde{X}(s))-b(X(s)))\\
&\quad  \mbox{}+\frac{1}{2}|\sigma(\widetilde{X}(s))-\sigma(X(s))|^2\cdot \alpha(\alpha+1)\cdot|\Delta_{s}|^{-\alpha-2}ds\Bigg]\\
&\quad \mbox{}+\mathbb{E}\left[\int_{0}^{t'}\int_{U_1}\{R(\Delta_{s}+\Delta_{c_1})-R(\Delta_{s})-DR(\Delta_{s})\cdot \Delta_{c_1}\}\nu_1(du)ds\right.\\
&\quad \left. \mbox{}+\int_{0}^{t'}\int_{U_2}\{R(\Delta_{s}+\Delta_{c_2})-R(\Delta_{s})\}\nu_2(du)ds\right]\\
&\mbox{}\leq |\Delta_0|^{-\alpha}+\mathbb{E}\left[\int_{0}^{t'}(\alpha+\frac{1}{2}\alpha(\alpha+1))\cdot\rho(|\Delta_{s}|^{-\alpha})ds\right]\\
&\quad \mbox{}+\mathbb{E}\left[\int_{0}^{t'}\int_{U_1}K\left(|\Delta_{s}|^{-\alpha}+|\Delta_{s}|^{-\alpha-1}|\Delta_{c_1}|\right)\nu_1(du)ds\right.\\
&\quad \left. \mbox{}+\int_{0}^{t'}\int_{U_2}K'\left(|\Delta_{s}|^{-\alpha}+|\Delta_{s}|^{-\alpha-1}|\Delta_{c_2}|\right)\nu_2(du)ds\right]\\
&\mbox{} \leq |\Delta_0|^{-\alpha}+\mathbb{E}\left[\int_{0}^{t'}K_1|\Delta_{s}|^{-\alpha}+K_2\rho(|\Delta_{s}|^{-\alpha})ds\right],
\end{align*}
where
$$
K'=\frac{1}{\delta^\alpha}(1+\alpha), \quad K_1=M(K+K'),\quad  K_2=\alpha+\frac{1}{2}\alpha(\alpha+1)+K+K'.
$$
Let $\rho_0(x)=K_1x+K_2\rho(x)$. Then $\rho_0$ is also concave and non-decreasing with $\rho_0(0)=0$. Moreover, there exists some $x_0\geq0$ such that,  for any $x\in [0,x_0]$, either $\rho(x)\geq x$ or $\rho(x)\leq x$ holds. Thus,
$$
\int_{0}^{x_0}\frac{1}{\rho_0(s)}ds\geq \int_{0}^{x_0}\frac{1}{(K_1+K_2)\max\{\rho(s),s\}}ds=+\infty
$$
and consequently,
$$\int_{0+}\frac{1}{\rho_0(s)}ds=+\infty.$$
Let $u(t')=E\left[|\Delta_{t'}|^{-\alpha}\right]$. Then, by Jensen's inequality,
$$
0\leq u(t')\leq |\Delta_0|^{-\alpha}+\int_{0}^{t}\rho_0(u(s\wedge{\tau_R}\wedge{T_{\frac{1}{n}}}))ds.
$$
Applying Lemma \ref{th-bihari},
$$
0\leq u(t')\leq \Omega^{-1}\left(\Omega(|\Delta_0|^{-\alpha})+t'\right)<+\infty.
$$
Letting $R\rightarrow\infty$ and by Fatou's lemma, one has
$$
0\leq u\left(t\wedge{T_{\frac{1}{n}}}\right)=\mathbb{E}\left[|\Delta_{t\wedge{T_{\frac{1}{n}}}}|^{-\alpha}\right]\leq \Omega^{-1}\left(\Omega(|\Delta_0|^{-\alpha}\right)+t\wedge{T_{\frac{1}{n}}})\leq\Omega^{-1}\left(\Omega(|\Delta_0|^{-\alpha})+t\right)<+\infty.
$$
Since $|\Delta_{T_{\frac{1}{n}}}|\leq \frac{1}{n}$ holds on the set $\{T_{\frac{1}{n}}<t\}$ and $R(x)=|x|^{-\alpha}$ is non-increasing for $x>0$, it follows that
$$
(\frac{1}{n})^{-\alpha}\mathbb{P}\{T_{\frac{1}{n}}<t\}\leq \mathbb{E}\left[|\Delta_{t\wedge{T_{\frac{1}{n}}}}|^{-\alpha}\cdot I_{\{T_{\frac{1}{n}}<t\}}\right]\leq \mathbb{E}\left[|\Delta_{t\wedge{T_{\frac{1}{n}}}}|^{-\alpha}\right]\leq \Omega^{-1}\left(\Omega(|\Delta_0|^{-\alpha})+t\right)<+\infty
$$
and so
$$
\mathbb{P}\{T_{\frac{1}{n}}<t\}\leq n^{-\alpha}\cdot{\Omega^{-1}\left(\Omega(|\Delta_0|^{-\alpha})+t\right)}.
$$
This shows that $\mathbb{P}\{T_0<t\}=0$ holds for all $t\geq0$. Thus, letting $t\rightarrow\infty$, we have $\mathbb{P}\{T_0<\infty\}=0$. In other words, $|\Delta_0|>0$ a.s. on the interval $[0,\infty)$, which completes the proof.
\end{proof}

\begin{remark}
If $\alpha=0$, our results reduces to Corollary 3.3 in \cite{Xi2017Jump}.
\end{remark}

\begin{example}
Consider the following SDE:
\begin{align}\label{exam 2}
X(t)=&X_0-\int_{0}^{t}[X^{3}(s)+X^{\frac{1}{3}}(s)]ds+\int_{0}^{t}2X(s)dB_s+\int_{0}^{t}\int_{|u|\leq1}\gamma|u| X(s)\widetilde{N_1}(ds,du)\nonumber\\
&+\int_{0}^{t}\int_{|u|>1}\gamma|u| X(s){N_2}(ds,du).
\end{align}
Here $\gamma$ is a positive constant such that $\int_{|u|\leq1}|\gamma u|^2\nu(du)=1$.  Note that
$$
xb(x)=-x^{4}-x^{\frac{4}{3}}
$$
and
\begin{align*}
(x-y)(b(x)-b(y))=&-(x-y)(x^{\frac{1}{3}}-y^{\frac{1}{3}}+x^{3}-y^{3})\\
\leq &-(x^{\frac{1}{3}}-y^{\frac{1}{3}})^2\left[(x^{\frac{1}{3}}+{\frac{1}{2}}y^{\frac{1}{3}})^2+{\frac{3}{4}}y^{\frac{1}{3}}\right]-(x-y)^2\left((x+{\frac{1}{2}}y)^2
+{\frac{3}{4}}y^2\right).
\end{align*}
We can check that the coefficients of \eqref{exam 2} satisfy Assumptions \ref{ass-b}, \ref{ass-a} and \ref{ass-confluence} for $\alpha=0$. Thus, employing Theorem \ref{existence}, we see that (\ref{exam 2}) has a unique non-explosive strong solution $X(t)$ for any initial value $X(0)=x \in \mathbb{R}$. According to Theorem \ref{nonconfluence}, we know that $X(t)$ has the non-confluent property.
\end{example}

\section{Conclusions}
\paragraph{}

This paper is devoted to study some qualitative properties of strong solutions for a class of JSDEs with the super linear growth and non-lipschitz conditions. We have obtained the non-explosive property of strong solutions for JSDEs with the super linear growth condition by applying similar arguments in \cite{fang2003stochastic}. By employing the  Bihari-Lasalle inequality, we have also established  the pathwise uniqueness of strong solutions to JSDEs with the non-Lipschitz condition, in which $\mathbb{E}[|\widetilde{X}(t)-X(t)|^\alpha]$ vanishes up to an appropriately defined stopping time by constructing a sequence of smooth functions. Moreover, we have showed the non-confluent property of strong solutions for JSDEs under some mild conditions.

These findings of the research have led the authors to the following main contributions: (i) it was relaxed for the usual linear growth condition which guarantees the non-explosive property of solutions; (ii) a generalized non-Lipschitz condition was given to guarantee the existence and uniqueness of the solution to the JSDEs; (iii) the method developed by \cite{FU2010306, Xi2017Jump} also works for uniqueness problem with respect to the JSDEs under the non-Lipschitz condition constructed in this paper, i.e., the non-Lipschitz condition in our paper has universality; (iv) non-confluent property of strong solutions to the JSDEs has been obtained under the nonlinear condition.

We would like to mention that JSDEs considered in this paper are all driven by Brownian motions and Poisson processes. It is well known that JSDEs driven by L\'{e}vy processes have attracted much attention recently (see, for example, \cite{Geman2007, He2018, Landis2013, Yang19}).  Therefore, it would be crucial and interesting to extend the results of this paper to JSDEs driven by general L\'{e}vy processes. We also note that various theoretical results with applications for SDEs driven by fractional Brownian motions (fBms) have been studied extensively in literature; for instance, we refer the reader to \cite{Biagini2008, bin2018, Czichowsky2018, Kang19, Yue18} and the references therein. Thus, it would be important to extend our results to JSDEs driven by fBms. We plan to address these problems as we continue our research.

\section*{Acknowledgements}
The authors are grateful to the editors and reviewers whose helpful comments and suggestions have led to much improvement of the paper.

%\end{CJK}
\end{document}